\documentclass[a4paper,11pt]{amsart}

\usepackage{amsfonts,latexsym,rawfonts,amsmath,amssymb,amsthm, mathrsfs, lscape}
\usepackage[all]{xy}
\usepackage[english]{babel}
\usepackage[utf8]{inputenc}

\usepackage{array, tabularx}

\usepackage{setspace}
\usepackage{textcomp}

\usepackage[hypertexnames=false,
backref=page,
    pdftex,
    pdfpagemode=UseNone,
    breaklinks=true,
    extension=pdf,
    colorlinks=true,
    linkcolor=blue,
    citecolor=blue,
    urlcolor=blue,
]{hyperref}

\usepackage[top=1.5in, bottom=.9in, left=0.9in, right=0.9in]{geometry}

\newcommand{\referenza}{}

\newtheorem{thm}{Theorem}[section]
\newtheorem*{thm*}{Theorem \referenza}
\newtheorem{cor}[thm]{Corollary}
\newtheorem*{cor*}{Corollary \referenza}

\newtheorem*{lem*}{Lemma \referenza}

\newtheorem{prop}[thm]{Proposition}
\newtheorem*{prop*}{Proposition \referenza}

\newtheorem*{conj*}{Conjecture \referenza}
\newtheorem{rmk}[thm]{Remark}

\newtheorem{defi}[thm]{Definition}

\numberwithin{equation}{section}

\def \N {\mathbb N}

\allowdisplaybreaks[1]

\title{Moduli space of families of positive $(n-1)$-weights}

\author{Simone Calamai}
\address[Simone Calamai]{Dipartimento di Matematica e Informatica ``Ulisse Dini''\\
Università di Firenze\\
via Morgagni 67/A\\
50134 Firenze, Italy}
\email{scalamai@math.unifi.it}
\email{simocala@gmail.com}

\thanks{The author is supported by the Project PRIN ``Varietà reali e complesse: geometria, topologia e analisi armonica'', 
by  the Simons Center for Geometry and Physics, Stony Brook University, by SIR 2014 AnHyC "Analytic aspects in complex and hypercomplex geometry" (code RBSI14DYEB),
and by GNSAGA of INdAM}

\date{\today}

\begin{document}

\begin{abstract}
We show the geometrical structure of the moduli space of positive-weighted trees with $n$ labels $1,\ldots , n$ which realize the same family of positive $(n-1)$-weights
and we characterize them as a family of positive multi-weights.
\end{abstract}

\maketitle

\section{Introduction}
 We are concerned in studying the problem about whether a family $F$ of positive real numbers is 
 the family of the $k$-weights of some positive weighted tree $\mathcal{T}$.
The next section gives the precise meaning of our aim; here we would like to recall some relevant results in the literature which inspired 
this note, and to present our contribution. 
 
 The first result that we would like to recall tracks back to $1965$, when Hakimi and Yau, in \cite{HakimiYau}, proved the following statement about graphs:
 $F$ is the family of the $2$-weights of a positive weighted graph if and only if the elements of $F$ satisfy the triangle inequalities, that is 
 \begin{align}\label{triangle inequalities}
\forall D_{i,j},\, D_{i,k},\, D_{k,j} \in F ,\quad D_{i,j} \leq D_{i,k}  + D_{j,k}  \; . 
 \end{align}
The role of the triangle inequalities is prominent as well in the criterion, given in \cite{Buneman,SimoesPereira,Zaretskii}, concerning trees:
a family $F$ of positive real numbers that satisfy the triangle inequalities is the family of the $2$-weights of a positive weighted tree if and only 
if for all $D_{i,j},\, D_{i,k},\, D_{i,h} ,\, D_{j,k} ,\, D_{j,h}, \, D_{k,h} \in F$ the maximum value in
\begin{align}\label{equa: four points condition}
 \{ D_{i,j} + D_{h,k},\, D_{i,k} + D_{h,j},\, D_{i,h} + D_{k,j}  \}
\end{align}
is attained at least twice; \eqref{equa: four points condition} is the well-known four points condition, and it can be stated also in order to
contain, as special instances, the triangle inequalities (see for example \cite{Baldisserri}).
 
Unlike the theorem by Hakimi and Yau, the result in \cite{Buneman,SimoesPereira,Zaretskii} also provides a statement on uniqueness; 
in fact for a $F$ satisfying the four points condition, there exists a unique positive-weighted tree $\mathcal{T}$ such that $F$ is the 
family of $2$-weights of $\mathcal{T}$ \cite{DHKMS}.
We can also state the uniqueness by saying that a positive-weighted tree $\mathcal{T}$ is determined by its family of $2$-weights.

Along the same vein, a theorem by Pachter and Speyer in \cite{PachterSpeyer} asserts that a positive weighted tree $\mathcal{T}$
with $n$ leaves $1,\ldots , n$ and no vertexes of degree $2$ is determined by its $k$-weights provided that $3\leq k \leq (n+1)/2$. 
A further generalization of the theorem by Pachter and Speyer has been given by Baldisserri and Rubei in \cite{BaldisserriRubei2}.

Recently, Baldisserri and Rubei in \cite{BaldisserriRubei1} were also able to provide a neat characterization on when a family $F$ of positive numbers is the family of the
$(n-1)$-weights of a positive weighted tree $\mathcal{T}$ of exactly $n\geq 3$ vertexes labeled with $1,\ldots , n$. 
(It is a mainly open the problem to characterize when a family of positive numbers is
the family of the $k$-weights of a weighted tree, and $3\leq k \leq n-2$, see for positive results \cite{HHMS,Rubei1,Rubei2}.)
 The condition found by Baldisserri and Rubei hinges upon the following inequality
 \begin{align}\label{equa: triangle general}
  (n-2) D_{ 1,\ldots, \hat i , \ldots , n } \leq  \sum_{j=1, j\neq i}^n D_{ 1,\ldots j , \ldots , n } \; .
 \end{align}
The inequality \eqref{equa: triangle general} can be seen as a generalization of the triangle inequalities \eqref{triangle inequalities},
but unlike the case of positive $2$-weights we don't expect a uniqueness to hold, as we have $k=n-1$ while the statement
of Pachter and Speyer holds for $k\leq (n+1)/2$. 

Whence, one aim of the present note is to study the moduli space of positive weights $w$ that can be put on a labeled tree $T$ so that
a given family $F$ of positive numbers is the family of $(n-1)$-weights of the positive weighted tree $(T,w)$;
the study of moduli spaces of trees has been addresses, among others, by \cite{BHV,PS}. 
The description is given in Proposition \ref{prop:moduli spaces}, and it has two main points of interest, we believe: the first is that
given a tree $T$ with $n$ labels $1,\ldots , n$, the moduli space (possibly empty) of positive weights essentially depends on how many
leaves has $T$ and on how many edges of $T$ do not have a leaf as an endpoint.
The second element of interest of Proposition \ref{prop:moduli spaces} is that when exactly  the equality holds in 
\eqref{equa: triangle general} exactly once in $F$, then $F$ is the family of  positive $(n-1)$-weights of a unique positive weighted tree $\mathcal{T}$;
thus in this very special case we recover a uniqueness phenomenon. Moreover, in this case the tree $\mathcal{T}$ is a star. 

As an application of Proposition \ref{prop:moduli spaces} we give an algebraic characterization of the equality case in 
\eqref{equa: triangle general} in Theorem \ref{thm: characterization}; 
the theorem suggests that it could be of interest to study families of positive 
numbers which are of multi-weights rather than of $k$-weights for a fixed $k$.
Other applications of our theory are the obstructions for a family $F$ of positive weights to be the family of $k$-weights of 
a positive-weighted tree.
 
\subsection{Acknowledgments.} The author is very grateful to Agnese Baldisserri and Elena Rubei for kindly introducing him into 
the topic of weighted graphs. He also wants to thank Xiuxiong Chen for constant support.

\section{Basic definitions and notation}

Given a tree $T$ (hereafter it is always meant to be finite), let $E(T), V(T)$, and $L(T)$ be respectively the set of edges, 
the set of vertexes, and the set of leaves of $T$. 
A positive-weighted tree $\mathcal{T} = (T, w)$ is a tree $T$ endowed with a positive valued function $w : E(T) \rightarrow (0,\, \infty)$.
\begin{defi}
 Let $\mathcal{T} = (T, w)$ be a positive-weighted tree, and pick a natural number $k\in\N$. 
 Consider $i_1 , \ldots, \, i_k \in V(T)$ and define
\begin{align*}
 D_{i_1 , \ldots, \, i_k } (\mathcal{T})= \min \{ w(S) \,|\, S \mbox{ is a connected sub-tree of } T \mbox{ with } 
 S\ni i_1 , \ldots, \, i_k \}
\end{align*}
to be a $k$-weight of $\mathcal{T}$.
\end{defi}

We are going to employ the following notation.

\begin{itemize}
 \item Let $\mathbb{R}_+ = \{ x \in \mathbb{R} \, | \, x > 0 \}$.
 \item Given a finite set $A$, then $\sharp A$ is the number of elements of $A$.
 \item For any $n\in \mathbb{N}$ with $n\geq1$, let $[n]= \{1,\,\ldots,\,n\}$.
 \item Given $n\in \mathbb{N},\, n\geq 3$ and $a,b \in [n]$, then $e(a,\,b)$ denotes the edge (if exists) connecting $a$ and $b$.
 \item Let $n\in \mathbb{N},\, n\geq 3$. Given a family of real numbers $F:= \{ D_I \}_{I\in {[n] \choose n-1 }}$, the element
 $D_{1,\ldots,\hat i , \ldots , n} \in F$ is also denoted as $D_{\hat i}$ for any $i \in [n]$.
 \item given $k\in \{ 1 , 2 \ldots \} $, for a $k$-dimensional simplex we mean an open simplex, that is
 \begin{align*}
  \left\{ (x_0 , \ldots , x_k) \in \mathbb{R}^{k+1} \mbox{ such that } 0< x_i <1 \mbox{ for all } i, \sum_{i=0}^k x_i < 1 \right\}. 
 \end{align*}
 \item $0$-dimensional simplex is a point.
 \item Given a weighted tree $(T, w)$, we denote with $w(T)$ the total weight of the tree, that is the finite sum of the weight of all the edges
 of $T$.
\end{itemize}

\begin{defi}
 Let $\ell$ be a leaf of a tree $T$. Let $e$ be the edge whose endpoint is $\ell$. Then we say that $e$ is a twig
 and we write $e=\mbox{twig}_\ell$. A non-twig-edge of $T$ is an edge whose neither endpoints are leaves.
\end{defi}

\begin{defi}
 We say that a tree (which is always meant to be finite) $T$ is labeled when all its leaves have attached a label, and there possibly are vertexes of $T$
which are not leaves and which are labeled.
\end{defi}

\begin{defi}
 Let $T$ be a tree and $a \in V(T)$ be a labeled vertex of $T$. We say that $a$ is a leaf-label if in fact $a \in L(T)$, while we
 say that $a$ is a non-leaf-label if $a \not\in L(T)$.
\end{defi}

\begin{defi}
 Let $T$ be a tree. We say that $T$ is a reduced labeled tree when, if $v \in V(T)$ and $\deg_T (v) = 2$, then $v$ is labeled.
\end{defi}

\begin{defi}
 Let $k,n \in \mathbb{N}$, $2\leq k < n$, and $F= \{ D_I\}_{I\in {[n] \choose k}}$ be a family of positive numbers.
We say that $F$ is positive treelike if there exists a labeled reduced, positive-weighted tree $\mathcal{T}=(T,w)$
such that 
\begin{align*}
 D_I (\mathcal{T}) = D_I, \quad \forall I \in  {[n] \choose k}\; .
\end{align*}
In this case we also say that $\mathcal{T}$ realizes $F$.

 Moreover, we say that $F$ is positive-leaf-treelike if there exists a labeled reduced, positive-weighted tree $\mathcal{A}$ with $n$ labels
 $[n]$ such that there are no non-leaf-labels.
\end{defi}

\section{Main results}

We start with recalling some definitions and results from \cite{BaldisserriRubei2} which we are going to use in the present section.

\begin{defi}
 Let $r\in \mathbb{N}-\{ 0 \}$. We say that a tree $P$ is an $r$-pseudostar if any edge of $P$ divides $L(P)$ into two sets such that at
 least one of them has cardinality less than or equal to $r$. Notice that a $1$-peudostar is a star.
 
 We say that a tree is essential if it has no vertexes of degree $2$.
\end{defi}

\begin{defi}
 Let $n,r\in\mathbb{N}-\{ 0 \}$. Let $\mathcal{T} = (T,w)$ be a weighted tree with $L(T)= [n]$  Let $e$
 be an edge of $T$ with weight $y$ and dividing $[n]$ into two sets such that each of them has strictly 
 more than $r$ elements. Contract $e$ and add $y/(n-r)$ to the weight of every twig of the tree $T$.
 We call this operation an $r$-IO operation on $\mathcal{T}$ and we call the inverse operation an $r$-OI operation.
\end{defi}

We now recall the result in \cite{BaldisserriRubei1} which is interesting to compare with our main result.

\begin{thm}[\cite{BaldisserriRubei1}]
 Let $n\in \mathbb{N}$, $n\geq 3$ and let $F:= \{ D_I\}_{I\in {[n] \choose n-1 }}$ be a family of positive real numbers.
  \begin{enumerate}
   \item There exists a positive-weighted tree $\mathcal{T}=(T,w)$ with at least $n$ vertexes, $1,\ldots , n$ that realizes $F$ if and only if
   \begin{align}\label{equa:disug baldiss rubei}
(n-2) D_{\hat i} \leq \sum_{j \in [n] - \{ i\}} D_{\hat j}
   \end{align}
for any $i \in [n]$ and at most one of the inequalities in \eqref{equa:disug baldiss rubei} is an equality.
   \item There exists a positive weighted tree $\mathcal{T} = (T,w)$ with only leaf-labels $1,\ldots, n$ 
   that realizes $F$ if and only if
   \begin{align*}
    (n-2) D_{\hat i} < \sum_{j \in [n] - \{ i\}} D_{\hat j}
   \end{align*}
for any $i \in [n]$.
   \end{enumerate}

\end{thm}

We now state the main proposition of this section.

\begin{prop}\label{prop:moduli spaces}
 Let $n\in\mathbb{N}, \, \geq 3$ and let $F:=\{ D_I \}_{I\in {[n] \choose n-1 }}$ be a family of positive real numbers with
 $D_{\hat 1} \leq\ldots \leq D_{\hat n}$ .  
 Let $M\in \{1,\, \ldots ,\, n\}$ be the number of distinct elements of $\{ D_I \}_{I\in {[n] \choose n-1 }}$ which have the maximum value.
 \begin{enumerate}
  \item \label{case0} Suppose that $M \in \{ n-1 , n\}$;  then the family $F$ is positive-treelike if and only if 
  \begin{align} \label{all strict inequalities}
   (n-2) D_{\hat i} < \sum_{j\in [n]-\{i\}} D_{\hat j}, \quad \forall i \in [n]. 
  \end{align}
Moreover, for any  labeled reduced tree $T$ which has only leaf-labels, the moduli space of positive weights $E(T)\to \mathbb{R}_+ $ which realize $F$
is a $N$-dimensional simplex, where $N$ is the number of non-twig-edges of $T$. In this case $F$ is not realized by other positive-weighted trees.
  \item\label{case1} Suppose that $2 \leq M\leq n-2$; then the family $F$ is positive-treelike if and only if 
  \eqref{all strict inequalities} holds.
Moreover, for any labeled reduced tree $T$ which has only leaf-labels, 
the moduli space of positive weights behaves as in the case (\ref{case0}); 
also, for any labeled reduced tree $T$ which has exactly $M$ non-leaf-labels and such that its leaves are labeled with $1,\ldots,n-M$, the moduli space of positive weights $E(T) \to \mathbb{R}_+$ 
which realize $F$ is given by an $N-1$-dimensional simplex, where $N\geq 1$ is the number of non-twig-edges of $T$.
Finally, there are no extra positive weighted trees which realize the family $F$.
 \item \label{case2} Suppose that $M=1$; then the family $F$ is positive-treelike if and only if either \eqref{all strict inequalities} holds, 
 (and in this the moduli spaces of positive weights $E(T) \to \mathbb{R}_+$ behaves exactly as in  the case (\ref{case0}))
 or
 \begin{align} \label{one equality}
   (n-2) D_{\hat c} = \sum_{j\in [n]-\{c\}} D_{\hat j} , \quad \exists ! c\in [n]
  \end{align}
 together with 
  \begin{align} \label{almost strict inequalities}
   (n-2) D_{\hat i} < \sum_{j\in [n]-\{i\}} D_{\hat j}, 	\quad \forall i\in [n]-\{c\}.
  \end{align}
In this case we observe a rigidity phenomenon: the family $F$ is
 realized by a unique positive weighted tree $\mathcal{T}=(T,w)$. More precisely, necessarily $c=n$, $T$ is a star whose center is labeled as $n$, 
 and the values of $w$ are given by
 \begin{align*}
  w(e(n,k)) = D_{\hat n} - D_{\hat k} \quad \forall k\in [n-1]\; .
 \end{align*}
 \end{enumerate}
\end{prop}

\begin{proof}
 We first claim that, if $F$ is realized by a positive weighted tree $\mathcal{T}=(T,w)$ and $T$ has at least a non-leaf-label, then the exact number of
 non-leaf-labels of $T$ is $M$. In fact, we have that if  $k$ is a non-leaf-label, then
 $D_{\hat k} = w(T)$, while if $h$ is a leaf-label, then $D_{\hat h} > w(T)$. 
 As consequence, if $M = \{ n-1 , n \}$, then $F $ can possibly be realized only by positive-weighted trees with $n$ leaves; otherwise,
 $F$ can be possibly realized by positive-weighted trees of either $n$ or $n-M$ leaves.
 Let us consider the case when $M\in \{ n-1 , n \}$. By means of \cite[Theorem 9 (b)]{BaldisserriRubei1}, we know that $F$ is treelike if and only if
 \eqref{all strict inequalities} holds. 
 Pick a tree $T$ with labels $[n]$. 
 We already argued that if $T$ has a non-leaf-label, then $T$ cannot be
 endowed with a positive weight that realizes $F$. 
 Whence, assume that $T$ has only leaf-labels; also, assume that $T$ is reduced and has $N\geq 0$ non-twigs-edges $e_1 , \ldots , e_N$,
 and exactly $n$ leaves $1,\ldots , n$.
 Let $(\tilde T , \tilde w)$ the unique positive-weighted essential $1$-pseudostar with $n$ leaves $1, \ldots , n$ that realizes $F$, by means of 
 \cite{BaldisserriRubei2}.
By how we ordered $D_{\hat 1}, \ldots , D_{\hat n}$, we observe that there holds
\begin{align*}
 \tilde w (\mbox{twig}_1) \geq \ldots \geq  \tilde w (\mbox{twig}_n) 
\end{align*}
and 
\begin{align*}
  \tilde w (\mbox{twig}_j) = \frac{1}{n-1} \cdot \left( -n D_{\hat j}+ \sum_{k=1 , k\neq j}^n D_{\hat k}\right) \; .
\end{align*}
We define the  weight $w:E(T) \to \mathbb{R}$ given by
\begin{align}\label{equa:definition of positive weight for N simplex}
 w(\mbox{twig}_j) := \tilde w (\mbox{twig}_j) - \frac{1}{n-1} \cdot \sum_{k=1}^{N} w(e_k) = 
 \frac{1}{n-1} \cdot \left(\left( \sum_{\ell=1 , \ell \neq j}^n D_{\hat \ell}\right) - nD_{\hat j} - \sum_{k=1}^{N} w(e_k)   \right), \; \forall j\in [n] 
\end{align}
and 
\begin{align}\label{equa: equation simplex open}
 w(e_\ell) \in (0,1), \, \sum_{\ell = 1}^N w(e_\ell) < D_{\hat 1} + \ldots + D_{\hat{n-1}} - nD_{\hat n} = (n-1)\tilde w (\mbox{twig}_n) .
\end{align}
we have to verify that the values of $w$ are actually positive and that $\mathcal{T}=(T,w)$ realizes $F$.

Since $(n-1)\tilde w (\mbox{twig}_n) > 0$, then it is always possible to choose $w(e_1), \ldots , w(e_N)$ such that they are all positive
and satisfy \eqref{equa: equation simplex open}.
Moreover, for $j\in [n]$ there holds
\begin{align*}
 w(\mbox{twig}_j) = \tilde w (\mbox{twig}_j) - \frac{1}{n-1} \cdot \sum_{k=1}^{N} w (e_k) > 
 \tilde w (\mbox{twig}_j) - \tilde w (\mbox{twig}_n) \geq 0.
\end{align*}
Whence, the weight $w:E(T) \to \mathbb{R}$ is in fact a positive-weight.

It remains to prove that $\mathcal{T}$ realizes $F$.
Let $j\in [n]$; then 
\begin{align*}
D_{\hat j} (\mathcal{T}) &= \sum_{k=1}^N  w(e_k) + \sum_{\ell = 1 , \ell \neq j}^n w(\mbox{twig}_\ell)\\
 &= \sum_{k=1}^N  w(e_k) + \sum_{\ell = 1 , \ell \neq j}^n \left( \tilde w(\mbox{twig}_\ell) - \frac{1}{n-1} \cdot \sum_{h=1}^N w(e_h)\right)\\
 &=  \sum_{\ell = 1 , \ell \neq j}^n \tilde w(\mbox{twig}_\ell) \\
 &= D_{\hat j} (\widetilde{ \mathcal{T}}) = D_{\hat j} \; . 
\end{align*}
Whence $\mathcal{T}$ realizes $F$.

It remains to show that given a positive weight $\bar w$ on $T$ which realizes $F$, then $\bar w$ has the structure as in
\eqref{equa:definition of positive weight for N simplex} and \eqref{equa: equation simplex open}.
 In fact, starting from $(T, \bar w)$ we can perform $N$ operations of type $1$-OI 
 contracting $e_1 , \ldots , e_N$.
 We thus obtain a sequence of weighted trees, for $j = 0 , \ldots , N$, $\mathcal{T}^{j} = (T^{(j)} , {\bar w}^{(j)})$ ,
 where we set $\mathcal{T}^{(0)} = (T , \bar w )$, ${\bar w}^{(j)}$ is a positive weight for all $j=0,\ldots , N$,
 $\mathcal{T}^{(j)}$ realizes $F$ for all $j=0,\ldots , N$, and $T^{(N)}$ is an essential $1$-pseudostar tree. 
 By means of \cite[Theorem 16 page 13]{BaldisserriRubei2}, it necessarily holds that $(T^{(N)} , {\bar w}^{(N)} ) = (\tilde T , \tilde w)$.
 Whence, reversing the $N$ operations $1$-IO performed, we get a way to obtain $(T, \bar w)$ from $(\tilde T , \tilde w)$; there follows that 
 the structure of $\bar w$ is necessarily as in \eqref{equa:definition of positive weight for N simplex} and \eqref{equa: equation simplex open}
 for suitable positive choices of $w(e_1), \ldots , w(e_N)$.
  
  Finally, it is evident from \eqref{equa: equation simplex open} that the moduli space of positive weights on $T$ which realize $F$ is
   an $N$-dimensional simplex.
  This concludes the proof of (\ref{case0}).
  
 About the case (\ref{case1}), if we start from a labeled reduced tree $T$ with no non-leaf-labels, the argument is in the same vein as in case (\ref{case0}).
  Thus let us consider a labeled reduced tree $T$ with some non-leaf-labels; as we argued at the beginning of the proof, in order to realize $F$ as a
  positive-tree, $T$ must have exactly $n-M$ leaves; by means of how we ordered $D_{\hat 1},\ldots , D_{ \hat n}$ we have to label the leaves 
  of $T$ with labels $1,\ldots , n-M$ (in any order). Now we claim that $T$ has at least a non-twig-edge, that is $N\geq 1$. Indeed,
  if we suppose by contradiction that all the edges of $T$ are twigs, then $T$, as a non-labeled tree, would have the structure of a $1$-pseudostar
  with exactly $n-M$ leaves; if $M > 1$, this immediately produces a contradiction.
  It remains the possibility that $M=1$, which falls in case (\ref{case2}). This completes the proof of the claim.

We consider the subfamily $\tilde F = \{ D_{\hat 1 }, \ldots , D_{\widehat{n-M}} \} \subset F$, and we regard
$D_{\hat 1 }, \ldots , D_{\widehat{n-M}}$ as a $(n-M-1)$-weights on $[n-M]$. We claim that $\tilde F$
is positive-leaf-treelike. By means of \cite{BaldisserriRubei1}, it suffices to prove that 
\begin{align} \label{equa: equa in claim 2 case 1}
 (n-M-2)D_{\hat j} < \sum_{k=1 , k\neq j}^{n-M} D_{k}\, , 
\end{align}
for all $ j\in [n-M]$. We suppose by contradiction that 
\begin{align*}
 (n-M-2) D_{n-M-2} \geq \sum_{k=1}^{n-M-1} D_{\hat k}\; .
\end{align*}
We know that
\begin{align*}
 (n-2) D_{\hat n} < D_{\hat 1} + \ldots + D_{\widehat{n-M}} +  D_{\widehat{n-M+1}} + \ldots +  D_{\widehat{n-1}}   
\end{align*}
 from which, since we ordered such that $ D_{\widehat{n-M+1}} = \ldots =  D_{\widehat{n-1}} = D_{\hat n} $,
 \begin{align*}
 (n-M-1) D_{\hat n} < D_{\hat 1} + \ldots + D_{\widehat{n-M-1}} +  D_{\widehat{n-M}} \leq (n-M-2+1) D_{\widehat{n-M}}    
 \end{align*}
 which gives a contradiction, being $D_{\hat n} = D_{\widehat{n-M}}$.
 Moreover, it is clear that if \eqref{equa: equa in claim 2 case 1} holds for $j=n-M$, then it also holds
 for all $j \in [n-M]$. This completes the proof of the present claim.
 
 As a consequence, we can now apply \cite[Theorem 16 page 12]{BaldisserriRubei2}, and we let 
 $\widetilde{\mathcal{T}} = (\tilde T , \tilde w)$ be the unique essential $1$-pseudostar tree with $n-M$ leaves $1,\ldots , n-M$
 positively weighted, that realizes $\tilde F$.
 
 The next claim is that the weight $w: E(T) \to \mathbb{R}$ given by
 \begin{align}\label{equa: positive weight case 1 part 1}
  w(\mbox{twig}_j) = \tilde w (\mbox{twig}_j)  - \frac{1}{n-M-1} \cdot \sum_{k=1}^N w(e_k) = D_{\hat n} - D_{\hat j}, \; 
  \forall j\in [n-M]  
 \end{align}
 and
 \begin{align}\label{equa: positive weight case 1 part 2}
  & w(e_\ell ) \in (0,1)\,  \forall \ell \in [N] \,, \\ \nonumber &\sum_{\ell = 1}^N w(e_\ell) = (n-M-1) (\tilde w( \tilde T) - D_{\hat n}) 
  = -(n-M-1)D_{\hat n} + \sum_{j=1}{n-M} D_{\hat j}
 \end{align}
is in fact a positive weight and it realizes $F$. We start with observing that $\tilde w (\tilde T) - D_{\hat n} > 0$; in fact 
\begin{align*}
 \tilde w (\tilde T) = \frac{1}{n-M-1} \cdot (D_{\hat 1} + \ldots + D_{\widehat{n-M}})  
\end{align*}
and whence the equation \eqref{equa: equa in claim 2 case 1} for $j=n$ precisely tells that 
$\tilde w (\tilde T) - D_{\hat n} > 0$.
Moreover, to check  that $w(\mbox{twig}_j)> 0$ for all $j \in [n-M]$, it is enough to observe that 
\begin{align*}
 \tilde w (\tilde T) - D_{\hat n} < \tilde w (\mbox{twig}_{n-M})\; .
\end{align*}
Indeed, $\tilde w (\mbox{twig}_{n-M}) = \tilde w (\tilde T) - D_{\widetilde{n-M}}$ and we ordered the $(n-1)$-weights so that
$D_{\hat n} > D_{\widetilde{n-M}}$.
Finally, we compute for all $j \in [n]$ the value of $D_{\hat j}((T,w))$; we start with the case $j\in [n-M]$ and we have
\begin{align*}
 D_{\hat j}((T,w)) &= \sum_{k=1}^{N} w(e_k) + \sum_{\ell=1 , \ell\neq j}^{n-M} w(\mbox{twig}_\ell) \\
 &= \sum_{\ell=1 , \ell\neq j}^{n-M} \tilde w(\mbox{twig}_\ell)\\
 &= D_{\hat j} ((\tilde T , \tilde w)) \\
 &= D_{\hat j} \; .
\end{align*}
 In the case when $j= n-M +1 ,\ldots , n$ then we compute
\begin{align*}
 D_{\hat j}((T,w)) &= w(T) =  \sum_{k=1}^{N} w(e_k) + \sum_{\ell=1 }^{n-M} w(\mbox{twig}_\ell) \\
 &= \sum_{\ell=1 , \ell\neq j}^{n-M} \tilde w(\mbox{twig}_\ell) - \frac{1}{n-M-1} \cdot \sum_{k=1}^N w(e_k) \\
 &= \tilde w(\tilde T) - ( \tilde w (\tilde T) - D_{\hat n} )  \\
 &= D_{\hat n} = D_{\hat j} \, ,
\end{align*}
for how we ordered the $(n-1)$-weights.

To conclude the proof of case (\ref{case1}), we notice that arguing as in case (\ref{case0}) we can show that 
if $\bar w$ is a positive weight on $T$ that realizes $F$, then $\bar w$ has the structure as in \eqref{equa: positive weight case 1 part 1}
and \eqref{equa: positive weight case 1 part 2}. It is also clear that such positive weights form an $(N-1)$-dimensional simplex.

Let us finally handle the case (\ref{case2}). Pick any labeled refined tree $T$ with only leaf-labels; then an argument in the spirit of case (\ref{case0})
gives the claimed conclusion in the case when \eqref{all strict inequalities} holds. 

Whence, it remains to discuss the case of a labeled reduced tree $T$ with some non-leaf-labels that would realize $F$ as a positive-weighted tree.
We notice one possible solution $\mathcal{A}= (A, p)$ of our problem is already discussed in \cite[Theorem 9 page 4]{BaldisserriRubei1}, 
so we want to prove that $\mathcal{A}$ is the unique reduced positive weighted tree which realizes $F$.
First of all, we claim that the only positive weight $w$ that one can put on the tree $A$ in order to realize $F$ is precisely $p$.
 In fact, for such  $w$ we have 
 \begin{align}
  w(\mbox{twig}_j) = D_{\hat n} - D_{\hat j} \; \forall j \in [n-1] \; .
 \end{align}
Moreover, since $E(A) = \{ \mbox{twig}_1 ,  \ldots , \, \mbox{twig}_{n-1} \}$, we have that 
$w : E(A) \to \mathbb{R}_{+}$ is uniquely determined by $\{ D_{\hat 1} , \ldots ,\, D_{\hat n} \}$, and whence $w$
must coincide with $p$.

Now we suppose that there exists a labeled reduced tree $T$, with $T\neq A$, which realizes the family $F$ by means of a positive weight $w$.
We aim to get to a contradiction. We start with claiming that, ordering the $(n-1)$-weights as
\begin{align*}
 D_{\hat 1} \leq  \ldots \leq D_{\widehat{n-1}} < D_{\hat n}\, , 
\end{align*}
then the label $n$ is placed on a vertex of $T$ which is not a leaf. In fact, if the label $n$ was placed on a leaf of $T$, then also the labels
$1, \ldots , n-1$ would be placed on leaves of $T$, and the fact that all the labels $1,\ldots , n$ are placed on leaves 
implies that 
\begin{align*}
 D_{\hat j} = w(T) - w(\mbox{twig}_j) 
\end{align*}
and hence 
\begin{align*}
 w(T) - w(\mbox{twig}_n) = D_{\hat n} = \frac{D_{\hat 1}+ \ldots + D_{\widehat{n-1}}}{n-2} 
 = \frac{(n-1) w(T) - \sum_{j=1}^{n-1} w(\mbox{twig}_j)}{n-2} > w(T)
\end{align*}
which gives a contradiction.

 As consequence of the above claim we have that $D_{\hat n} = w(T)$; in fact, $L(T) \subset [n-1]$ and $D_{\hat n} = D_{1,\ldots , n-1}$.
 
 The next claim is that the labels from $1$ to $n-1$ are placed on leaves of $T$. I fact, if a label $k \in [n-1]$ would be placed
 on a non-leaf vertex of $T$, then one would have $D_{\hat k} = D_{\hat n}$, against our assumption $M=1$.
 
 Now we claim that all the edges of $T$ are twigs. In fact, if by contradiction $T$ would have a non-twig-edge of weight, say, $x>0$, then 
 \begin{align*}
  w(T) \geq w(\mbox{twig}_1) + \ldots + w{ \mbox{twig}_{n-1}} + x
 \end{align*}
 but 
\begin{align*}
  w(T) = D_{\hat n} = \frac{\sum_{j=1}^{n-1} D_{\hat j} }{n-2} = \frac{(n-1)w(T) - \sum_{1}^{n-1} w(\mbox{twig}_j)}{n-2} \geq w(T) + \frac{x}{n-2} > w(T)
 \end{align*}
which gives a contradiction.
Summarizing,  $T$ as labeled reduced tree has $n-1$ leaves $1,\ldots , n-1$, the label $n$ is placed on a non-leaf vertex,
and all the edges of $T$ are twigs; this amounts to say that $T=A$ and shows the contradiction we were looking for.
This completes the discussion of the case (\ref{case2}) and of all the proposition.

\end{proof}

\begin{rmk}
 Notice that Proposition \ref{prop:moduli spaces} can be seen as a quantitative version of \cite[Theorem 9 at page 4]{BaldisserriRubei1}.
 In particular the rigidity phenomenon emphasized in the statement of Proposition \ref{prop:moduli spaces} will be of particular interest in 
 the present note.
\end{rmk}

\begin{rmk}
 It would be nice to get a geometrical description of the moduli space of all the positive-weighted trees 
 corresponding to a given family $F$ of positive $(n-1)$-weights. In fact, Proposition \ref{prop:moduli spaces} rather gives the moduli spaces of 
 positive-weights that can be put on a tree in order to realize $F$. 
\end{rmk}

\begin{cor}\label{cor: corollary n-1 case}
 Let $n\in\mathbb{N}, \, \geq 3$ and let $F:=\{ D_I \}_{I\in {[n] \choose n-1 }}$ be a family of positive real numbers for which  there hold 
 \eqref{one equality} and \eqref{almost strict inequalities}. Then the family $F$ is equivalent to the family of positive two-weights
 $F_0 = \{ D_{ab}\}_{1\leq a < b\leq n}$ for which 
 \begin{align*}
  D_{ij} = D_{ci} + D_{cj}, \quad \exists c\in[n]\forall i,j\in[n]-\{c\}.
 \end{align*}
The correspondence between the two families is the following
 \begin{align}\label{equa: correspondence1}
  D_{\hat i} &= \sum_{j\in [n]-\{i,c\}}D_{cj}, \quad \forall i\in [n]-\{ c \} \\
 \nonumber D_{\hat c} &= \sum_{j\in [n]-\{c\}}D_{cj}
 \end{align}
 and, conversely,
 \begin{align}\label{equa:correspondence2}
  D_{ic} &= D_{\hat c} - D_{\hat i}, \quad \forall i\in[n]-{c} \\
  \nonumber D_{ij} &= 2 D_{\hat c} - D_{\hat i} - D_{\hat j}, \quad \forall i \neq j\in[n]-{i,j}. 
 \end{align}

\end{cor}

\begin{proof}
 Given $F$ as in the statement, by means of Proposition \ref{prop:moduli spaces} we know that $F$ is realized only by the positive-weighted tree
 $\mathcal{T} = (T,w)$ 
 such that $T$ is a star whose center is labeled as $c$,  and the values of $w$ are given by
 \begin{align*}
  w(e(c,k)) = D_{\hat c} - D_{\hat k} \quad \forall k\in [n]-\{c\}.
 \end{align*}
 Whence, it is easy to see that the positive-weighted tree $\mathcal{T}$ gives the family of positive two-weights $F_0$. 
 
 Conversely, given the family $F_0$, it is now clear that one positive-weighted tree that realizes $F_0$ is $\mathcal{T}$ as here above.
 Moreover, thanks to \cite{DHKMS} we know that any family of positive two-weights can be realized by at most a unique 
 positive-weighted tree. Whence the families $F$ and $F_0$ can be put in correspondence and it can be easily verified that the correspondence is
 exactly that one expressed in both \eqref{equa: correspondence1} and \eqref{equa:correspondence2}. This completes the proof of the corollary.
 \end{proof}
 
 \begin{rmk}
  The interest of Corollary \ref{cor: corollary n-1 case} is that now we can give to the family of positive $(n-1)$-weights $F$ two different
  interpretations: the first one is geometrical, and it is captured by the description of the positive-weighted star tree $\mathcal{T}$ which realizes
  $F$. The second one has an algebraic flavor: we can replace the family $F$ with a family of positive two-weights. Notice that, by means of
  Proposition \ref{prop:moduli spaces}, such a correspondence does not exist for any other family of positive $(n-1)$-weights.
  Notice that by means of \cite{PachterSpeyer}, we could state in Corollary \ref{cor: corollary n-1 case} a correspondence between 
  the family $F$ and a family $F_k$ of positive $k$-weights, where $3\leq k \leq (n+1)/2$. 
 \end{rmk}

\begin{thm}\label{thm: characterization}
 Let $m,k\in \mathbb{N}, 2\leq k < m $,  and $F= \{D_I\}_{I \in {[m] \choose k}}$ be a family of positive real numbers.
 Suppose that 
 \begin{align} \label{equa: subequality}
  (k-1) D_{a_1 , \ldots , \hat a_{k+1} } = \sum_{j=1}^k D_{a_1, \ldots , \hat a_j , \ldots , a_{k+1}}, \quad \exists a_1 , \ldots a_{k+1} \in [m]. 
 \end{align}
Then the family $F$ is positive-treelike if and only if the family 
\begin{align}\label{family mixed}
 F_1 &:= F \cup \{ D_{a_i, a_j}\}_{1\leq i < j \leq k+1} \\
 \nonumber &D_{a_i, a_{k+1}} = D_{a_1 , \ldots , \hat a_{k+1} } - D_{a_1, \ldots , \hat a_i , \ldots , a_{k+1}}, \quad \forall i\in [k]\\
 \nonumber &D_{a_i, a_j} = 2D_{a_1 , \ldots , \hat a_{k+1} } - D_{a_1, \ldots , \hat a_i , \ldots , a_{k+1}}- D_{a_1, \ldots , \hat a_j , \ldots , a_{k+1}},
 \quad \forall i<j\in [k]
\end{align}
is positive-treelike, and moreover they are realized by exactly the same positive-weighted trees.
\end{thm}

\begin{proof}
 $\Rightarrow$ Suppose that the family $F$ is positive treelike, realized by a positive-weighted tree $\mathcal{A}=(A,p)$.
 Also, the subfamily $F_{|{a_1 , \ldots , a_{k+1}}}$ is treelike, and by means of \eqref{equa: subequality} and of Proposition \ref{prop:moduli spaces},
 the only positive-weighted tree that realizes it is the star tree $\mathcal{T}=(T,w)$.
 Now, it is clear that  $\mathcal{A}_{|\{a_1 ,\ldots , a_{k+1}\}} = \mathcal{T}$; moreover, by means of Corollary \ref{cor: corollary n-1 case} applied 
 with $n=k+1$, the subtree $\mathcal{A}_{|\{a_1 ,\ldots , a_{k+1}\}}$ is given by the family of positive two-weights described in \eqref{family mixed}.
 Whence, the family $F_1$ is as well realized by the positive-weighted tree $\mathcal{A}=(A,p)$.
 
 $\Leftarrow$ Suppose that the family $F_1$ is positive-treelike, realized by a positive-weighted tree $\mathcal{A}=(A,p)$.
 By means of \cite{DHKMS} the subfamily of positive two-weights as in \eqref{family mixed} is realized by a unique positive-weighted
 tree $\mathcal{T}=(T,w)$; by Corollary \ref{cor: corollary n-1 case}, the information about the family of positive two-weights in \eqref{family mixed}
 is encoded in the subfamily of positive $k$-weights $\{ D_{a_1, \ldots , \hat a_j , \ldots , a_{k+1}}\}_{j=1, \ldots , k+1} \subseteq F$.
 Whence, the family $F$ is as well realized by the positive-weighted tree $\mathcal{A}=(A,p)$.
 This concludes the proof of the theorem.
\end{proof}


\begin{thebibliography}{10}

\bibitem{Baldisserri}
A. Baldisserri, Buneman's theorem for trees with exactly $n$ vertices, arXiv:1407.0048v1. 

\bibitem{BaldisserriRubei1}
 A. Baldisserri, E. Rubei, On graphlike $k$-dissimilarity vectors, Annals of Combinatorics, \textbf{18} (2014), no. 3, 365-381. arXiv:1211.0423v2.

\bibitem{BaldisserriRubei2}
A. Baldisserri, E. Rubei, Treelike families of multiweights, arXiv:1404.6799. 

\bibitem{BHV}
L.J. Billera, S.P. Holmes, Vogtman, Geometry of the Space of Phylogenetic Trees, Advances in Applied Mathematics, \textbf{27} (2001), no.4, 733–767.

\bibitem{Buneman} 
P. Buneman, A note on the metric properties of trees, Journal of Combinatorial Theory, Ser B., \textbf{17} (1974), 48-50.

\bibitem{DHKMS}
A. Dress, K.T. Huber, J. Koolen, V. Moulton, A. Spillner, Basic phylogenetic combinatorics. Cambridge University Press, Cambridge, 2012.

\bibitem{HakimiYau}
 S.L. Hakimi, S.S Yau, Distance matrix of a graph and its realizability, Quart. Appl. Math., \textbf{22}, (1965), 305-317. 
 
\bibitem{HHMS}
 S. Hermann, K. Huber, V. Moulton, A. Spillner, Recognizing treelike dissimilarities, Journal of Calssification 29 (2012), (no. 3), 321-340.

\bibitem{PachterSpeyer}
L. Pachter, D. Speyer, Reconstructing trees from subtree weights, Appl. Math. Lett., \textbf{17}, (2004), no.6, 615-621.

\bibitem{PS}
L. Pachter, B. Sturmfels, The tropical Grassmannian, Advances in Geometry, \textbf{4} (2004), 389-411.

\bibitem{Rubei1}
E. Rubei, Sets of double and triple weights of trees, Annals of Combinatorics 15 (2011), no. 4, 723-734.

\bibitem{Rubei2}
E. Rubei, On dissimilarity vectors of general weighted trees, Discrete Mathematics 312 (2012), no. 19, 2872-2880.

\bibitem{SimoesPereira} 
 J.M.S. Simoes Pereira, A note on the tree realizability of a distance matrix, Journal of Combinatorial Theory, \textbf{6}, (1969), 303-310.
 
\bibitem{Zaretskii} K.A. Zaretskii, Constructing trees from the set of distances between pendant vertices,
Uspehi Matematiceskih Nauk., \textbf{20}, (1965), 90-92.
 

\end{thebibliography}
\end{document}